\documentclass[12pt]{article}

\pdfoutput=1

\usepackage[utf8]{inputenc}
\usepackage[T1]{fontenc}
\usepackage[english]{babel} 
\usepackage{csquotes}

\usepackage{graphicx,color,nicefrac,enumerate,listings}
\usepackage{amsmath,amsfonts,amsthm,amssymb}
\usepackage{mathtools} 
\usepackage[retainorgcmds]{IEEEtrantools}
\usepackage{dsfont}
\usepackage{bigints}

\usepackage{geometry}
\numberwithin{equation}{section}
\usepackage{hyperref}

\newcommand{\+}{\mkern1mu}
\newcommand{\tp}{\mkern2mu}
\newcommand{\twovector}[2]{\begin{pmatrix} #1 \\ #2 \end{pmatrix}}
\renewcommand\yesnumber{\refstepcounter{equation}\tag{\theequation}} 

\newcommand{\ud}{\,\mathrm{d}}	
\newcommand{\dx}{\ud x}
\newcommand{\ds}{\ud s}

\newcommand{\dWs}{\ud W_s}

\newcommand{\e}{\mathrm{\+ e}}
\newcommand{\N}{\mathbb{N}}
\newcommand{\R}{\mathbb{R}}

\DeclareMathOperator{\E}{\mathbb{E}}
\renewcommand{\P}{\mathbb{P}}
\newcommand{\bbone}{\mathds{1}}  

\newcommand{\bbH}{\mathbb{H}}

\newcommand{\cF}{\mathcal{F}}

\newcommand{\cP}{\mathcal{P}}

\newcommand{\bfA}{\mathbf{A}}
\newcommand{\bfB}{\mathbf{B}}

\newcommand{\bfF}{\mathbf{F}}

\newcommand{\bfH}{\mathbf{H}}

\newcommand{\bfP}{\mathbf{P}}

\newcommand{\bfX}{\mathbf{X}}

\DeclarePairedDelimiter{\abs}{\lvert}{\rvert}
\DeclarePairedDelimiter{\norm}{\lVert}{\rVert}
\DeclareMathOperator{\CovOp}{CovOp}
\DeclareMathOperator{\Var}{Var}
\DeclareMathOperator{\Cov}{Cov}
\DeclareMathOperator{\id}{id}

\theoremstyle{definition}
\newtheorem{Def}{Definition}[section]
\theoremstyle{plain}
\newtheorem{lemma}[Def]{Lemma}
\newtheorem{theorem}[Def]{Theorem}
\newtheorem{cor}[Def]{Corollary}
\newtheorem{prop}[Def]{Proposition}
\theoremstyle{remark}

\usepackage[backend=bibtex,style=numeric,firstinits=true,sorting=nty,maxbibnames=99]{biblatex}
\DeclareNameAlias{default}{last-first}

\AtBeginBibliography{%
}

\DefineBibliographyStrings{english}{%
	andothers = {\addcomma\addspace\textsc{et\addabbrvspace al}\adddot},
	and = {\textsc{and}}
}

\DeclareFieldFormat
[article,inbook,incollection,inproceedings,patent,thesis,unpublished]
{title}{#1}

\renewbibmacro{in:}{%
	\ifentrytype{article}{%
	}{%
	\printtext{\bibstring{in}\intitlepunct}%
}%
}

\renewbibmacro*{volume+number+eid}{%
	\printfield{volume}%
	\setunit*{\addcomma\space}%
	\printfield{number}%
	\setunit{\addcomma\space}%
	\printfield{eid}}

\DeclareFieldFormat{pages}{#1}

\renewbibmacro*{publisher+location+date}{%
	\printlist{publisher}%
	\setunit*{\addcomma\space}%
	\printlist{location}%
	\setunit*{\addcomma\space}%
	\usebibmacro{date}%
	\newunit}

\begin{document}

\title{Lower bounds for weak approximation errors \\
		for spatial spectral Galerkin approximations \\
		of stochastic wave equations}
\author{Ladislas Jacobe de Naurois, Arnulf Jentzen, and Timo Welti\\[2mm]
\emph{ETH Z\"{u}rich, Switzerland}}
\date{\today}
\maketitle

\begin{abstract}
Although for a number of semilinear stochastic wave equations existence and uniqueness results for corresponding solution processes are known from the literature, these solution processes are typically not explicitly known and numerical approximation methods are needed in order for mathematical modelling with stochastic wave equations to become relevant for real world applications.
This, in turn, requires the numerical analysis of convergence rates for such numerical approximation processes.
A recent article by the authors proves upper bounds for weak errors for spatial spectral Galerkin approximations of a class of semilinear stochastic wave equations.
The findings there are complemented by the main result of this work, that provides lower bounds for weak errors which show that in the general framework considered the established upper bounds can essentially not be improved.
\end{abstract}


\section{Introduction}

In this work we consider numerical approximation processes of solution processes of stochastic wave equations and examine corresponding weak convergence properties.
As opposed to strong convergence, weak convergence even in the case of stochastic evolution equations with regular nonlinearities are still only poorly understood (see, e.g., \cite{Hausenblas2010, KovacsLindnerSchilling2014, KovacsLarssonLindgren2012, KovacsLarssonLindgren2013, Wang2015} for several weak convergence results for stochastic wave equations and, e.g., the references in Section~1 in \cite{DeNauroisJentzenWelti2015ArXiv} for further results on weak convergence in the literature).
Therefore, equations available to current numerical analysis are limited to model problems such as the ones considered in the present article that cannot take into account the full complexity of models for evolutionary processes under influence of randomness appearing in real world applications (see, e.g., the references in Section~1 in \cite{DeNauroisJentzenWelti2015ArXiv}).
The recent article~\cite{DeNauroisJentzenWelti2015ArXiv} by the authors provides upper bounds for weak errors for spatial spectral Galerkin approximations of a class of semilinear stochastic wave equations, including equations driven by multiplicative noise and, in particular, the hyperbolic Anderson model.
The purpose of this work is to show that the weak convergence rates for stochastic wave equations established in Theorem~1.1 in \cite{DeNauroisJentzenWelti2015ArXiv} can in the general setting there \emph{essentially not be improved}.
This is achieved by proving lower bounds for weak errors in the case of concrete examples of stochastic wave equations with additive noise and without drift nonlinearity (see Corollary~\ref{cor:lower_bound_exp} below).
We argue similarly to the reasoning in Section~6 in Conus et al.~\cite{ConusJentzenKurniawan2014} and Section~9 in Jentzen \& Kurniawan~\cite{JentzenKurniawan2015}.
First results on lower bounds for strong errors for two examples of stochastic heat equations were achieved in Davie \& Gaines~\cite{DavieGaines2001}.
Furthermore, lower bounds for strong errors for examples and whole classes of stochastic heat equations have been established in M\"uller-Gronbach et al.~\cite{Muller-GronbachRitterWagner2008} (see also the references therein) and in M\"uller-Gronbach \& Ritter~\cite{Muller-GronbachRitter2007}, respectively.
Results on lower bounds for weak errors in in the case of a few specific examples of stochastic heat equations can be found in Conus et al.~\cite{ConusJentzenKurniawan2014} and in Jentzen \& Kurniawan~\cite{JentzenKurniawan2015}.

\begin{theorem} \label{thm:lowerbound}
For all real numbers $ \eta, T \in ( 0, \infty ) $,
every $ \R $-Hilbert space
$ ( H , \langle \cdot, \cdot \rangle_{ H } , \norm{\cdot}_{ H } ) $,
every probability space $ ( \Omega , \cF , \P ) $
with a normal filtration $ ( \mathbb{F}_t )_{ t \in [ 0 , T ] } $,
every $ \id_H $-cylindrical $ ( \Omega , \cF , \P, ( \mathbb{F}_t )_{ t \in [ 0 , T ] } ) $-Wiener process
 $ ( W_t )_{ t \in [ 0 , T ] } $,
and
every orthonormal basis $ ( e_n )_{ n \in \N } \colon \N \to H $ of $ H $
there exist
an increasing sequence
$ ( \lambda_n )_{ n \in \N } \colon \N \to (0,\infty) $,
a linear operator
$ A \colon D(A) \subseteq H \to H $
with
$ D(A) = \bigl\{ v \in H \colon \sum_{ n \in \N } \abs{ \lambda_n \langle e_n,v \rangle_{H} }^2 < \infty \bigr\} $
and
$ \forall \, v \in D(A) \colon A v = \sum_{ n \in \N } - \lambda_n \langle e_n,v \rangle_{ H } e_n $,
a family of interpolation spaces
$ ( H_r , \langle \cdot , \cdot \rangle_{ H_r } , \norm{\cdot}_{ H_r } )$, $ r \in \R $,
associated to $ - A $ (cf., e.g., \cite[Section~3.7]{SellYou2002}),
a family of $ \R $-Hilbert spaces $ ( \bfH_r , \langle \cdot , \cdot \rangle_{ \bfH_r } , \norm{\cdot}_{ \bfH_r } )$, $ r \in \R $,
with
$ \forall \,  r \in \R \colon ( \bfH_r , \langle \cdot , \cdot \rangle_{ \bfH_r } , \norm{\cdot}_{ \bfH_r } ) = \bigl( H_{ \nicefrac{r}{2} } \times H_{ \nicefrac{r}{2} - \nicefrac{1}{2} }, \langle \cdot, \cdot \rangle_{ H_{ \nicefrac{r}{2} } \times H_{ \nicefrac{r}{2} - \nicefrac{1}{2} } }, \norm{ \cdot }_{ H_{ \nicefrac{r}{2} } \times H_{ \nicefrac{r}{2} - \nicefrac{1}{2} } } \bigl) $,
families of functions
$ P_N \colon \bigcup_{ r \in \R } H_r \to \bigcup_{ r \in \R } H_r $, $ N \in \N \cup \{ \infty \} $,
and
$ \bfP_N \colon \bigcup_{ r \in \R } \bfH_r \to \bigcup_{ r \in \R } \bfH_r $, $ N \in \N \cup \{ \infty \} $,
with
$ \forall \, N \in \N \cup \{ \infty \}, r \in \R, u \in H_r, ( v , w ) \in \bfH_r \colon
\bigl( P_N (u) =  \sum_{ n = 1 }^N \langle (\lambda_n)^{-r} e_n , u \rangle_{ H_r } (\lambda_n)^{-r} e_n
\text{ and }
\bfP_N ( v , w ) = ( P_N ( v ) , P_N ( w ) ) \bigr) $,
a linear operator $ \bfA \colon D ( \bfA ) \subseteq \bfH_0 \to \bfH_0 $
with
$ D ( \bfA ) = \bfH_1 $
and
$ \forall \, (v,w) \in \bfH_1 \colon \bfA( v , w ) = ( w , A v ) $,
real numbers
$ \gamma, c \in ( 0, \infty ) $,
and functions
$ \xi \in \mathcal{L}^2 ( \P \vert_{\mathbb{F}_0} ; \bfH_{ \gamma } ) $,
$ \varphi \in C^{2}_\mathrm{b}( \bfH_0 , \R ) $,
$ \bfF \in C_{ \mathrm{b} }^2( \bfH_0, \bfH_0 ) $,
$ \bfB \in C_{ \mathrm{b} }^2( \bfH_0, \mathrm{HS}( H, \bfH_0 ) ) $
and $ ( C_\varepsilon )_{ \varepsilon \in ( 0, \infty ) } \colon ( 0, \infty ) \to [ 0, \infty ) $
with
$ \forall \, \beta \in ( \nicefrac{\gamma}{2}, \gamma ] \colon ( -A )^{ - \nicefrac{\beta}{2} } \in \mathrm{HS}( H_0 ) $,
$ \bfF \in C^{2}_\mathrm{b}( \bfH_0, \bfH_{ \gamma } ) $,
and
$ \bfB \in C^{2}_\mathrm{b}( \bfH_0, L( H, \bfH_\gamma ) ) $
such that
\begin{enumerate}[(i)]
\item
it holds that there exist
up to modifications unique
$( \mathbb{F}_t )_{ t \in [0,T] } $-predictable stochastic processes
$ \bfX^N \colon [ 0 , T ] \times \Omega \to \bfP_N ( \bfH_0 ) $, $ N \in \N \cup \{ \infty \} $,
which satisfy
for all $ N \in \N \cup \{ \infty \} $, $ t \in [ 0 , T ] $ that
$ \sup_{ s \in [ 0 , T ] } \norm{ \bfX_s^N }_{ \mathcal{L}^2 ( \P ; \bfH_0 ) } < \infty $
and
$ \P $-a.s.\ that
\begin{equation} \label{eq:intro,2}
\bfX_t^N =  \e^{t \bfA} \bfP_N \xi + \int_0^t \e^{ (t-s)\bfA } \bfP_N \bfF( \bfX_s^N ) \ds  + \int_0^t \e^{ (t-s)\bfA } \bfP_N \bfB( \bfX_s^N ) \dWs
\end{equation}
\item
and it holds
for all $ \varepsilon \in ( 0, \infty ) $, $ N \in \N $ that
\begin{equation} \label{eq:ratebound}
c \cdot ( \lambda_N )^{ - \eta } \leq \abs[\big]{ \E \bigl[ \varphi \bigl( \bfX_T^\infty \bigr)\bigr] - \E \bigl[\varphi \bigl( \bfX_T^N \bigr) \bigr] } \leq C_\varepsilon \cdot ( \lambda_N )^{ \varepsilon - \eta }.
\end{equation}
\end{enumerate}
\end{theorem}
Here and below we denote
for every non-trivial $ \R $-Hilbert space
$ ( V , \langle \cdot, \cdot \rangle_{ V } , \norm{\cdot}_{ V } ) $
and every $ \R $-Hilbert space
$ ( W , \langle \cdot, \cdot \rangle_{ W } , \norm{\cdot}_{ W } ) $
by
$ C^{2}_\mathrm{b}( V , W ) $
the set of all
globally bounded twice Fr\'echet differentiable functions from $ V $ to $ W $
with globally bounded derivatives.
Theorem~1.1 is a direct consequence of
Theorem~1.1 in \cite{DeNauroisJentzenWelti2015ArXiv}
(with
$ \gamma = 2 \eta $,
$ \beta = \min\{ \eta + \varepsilon, 2 \eta \} $,
$ \rho = 0 $
in the notation of Theorem~1.1 in \cite{DeNauroisJentzenWelti2015ArXiv})
and Corollary~\ref{cor:lower_bound_exp} below
(with
$ p = \nicefrac{1}{\eta} $,
$ \delta = \nicefrac{1}{2} - \eta $
in the notation of Corollary~\ref{cor:lower_bound_exp} below).
Inequality~\eqref{eq:ratebound} reveals that the weak convergence rates in Theorem~1.1 in \cite{DeNauroisJentzenWelti2015ArXiv} are essentially sharp.
More details and further lower bounds for weak approximation errors for stochastic wave equations can be found in Corollary~\ref{cor:lower_bound_exp} below.
%
%
%
\section{Lower bounds for weak errors} \label{sec:lower_bounds}
\subsection{Setting} \label{subsec:setting_lower}
Let $ ( H , \langle \cdot, \cdot \rangle_{ H } , \norm{\cdot}_{ H } ) $ be a separable $ \R $-Hilbert space,
for every set $ A $ let $ \cP( A ) $  be the power set of $ A $,
let $ T \in ( 0 , \infty ) $, let $ ( \Omega , \cF , \P ) $ be a probability space with a normal filtration $ ( \mathbb{F}_t )_{ t \in [ 0 , T ] } $, let $ ( W_t )_{ t \in [ 0 , T ] } $ be an $ \id_H $-cylindrical $ ( \Omega , \cF , \P, ( \mathbb{F}_t )_{ t \in [ 0 , T ] } ) $-Wiener process, let $ \bbH \subseteq H $ be a non-empty orthonormal basis of $ H $, let $ \lambda \colon \bbH \to \R $ be a function with $ \sup_{ h \in \bbH } \lambda_{h} < 0 $, let $ A \colon D(A) \subseteq H \to H $ be the linear operator which satisfies $ D(A) = \bigl\{ v \in H \colon \sum_{ h \in \bbH } \abs{ \lambda_h \langle h, v \rangle_H }^2 < \infty \bigr\} $ and $ \forall \, v \in D(A) \colon A v = \sum_{ h \in \bbH} \lambda_{ h } \langle h,v \rangle_{ H } h $, let $ ( H_r , \langle \cdot , \cdot \rangle_{ H_r } , \norm{\cdot}_{ H_r } )$, $ r \in \R $, be a family of interpolation spaces associated to $ - A $, let $ ( \bfH_r , \langle \cdot , \cdot \rangle_{ \bfH_r } , \norm{\cdot}_{ \bfH_r } )$, $ r \in \R $, be the family of $ \R $-Hilbert spaces which satisfies for all $ r \in \R $ that $ ( \bfH_r , \langle \cdot , \cdot \rangle_{ \bfH_r } , \norm{\cdot}_{ \bfH_r } ) = \bigl( H_{ \nicefrac{r}{2} } \times H_{ \nicefrac{r}{2} - \nicefrac{1}{2} }, \langle \cdot, \cdot \rangle_{ H_{ \nicefrac{r}{2} } \times H_{ \nicefrac{r}{2} - \nicefrac{1}{2} } }, \norm{ \cdot }_{ H_{ \nicefrac{r}{2} } \times H_{ \nicefrac{r}{2} - \nicefrac{1}{2} } } \bigl) $, let $ P_{I} \colon \bigcup_{r\in\R} H_r \to \bigcup_{r\in\R} H_r$, $ I \in \cP( \bbH ) $, and $ \bfP_{I} \colon \bigcup_{ r \in \R } \bfH_r \to \bigcup_{ r \in \R } \bfH_r $, $ I \in \cP( \bbH ) $, be the functions which satisfy for all $ I \in \cP( \bbH ) $, $ r \in \R $, $ u \in H_r $,  $ ( v , w ) \in \bfH_r $ that $ P_{I} (u) =  \sum_{ h \in I }  \langle \abs{\lambda_h}^{-r} h , u \rangle_{ H_r } \abs{\lambda_h}^{-r} h $ and $ \bfP_{I} ( v , w ) = \bigl( P_{I} ( v ) , P_{I} ( w ) \bigr) $, let $ \bfA \colon D ( \bfA ) \subseteq \bfH_0 \to \bfH_0 $ be the linear operator which satisfies $ D ( \bfA ) = \bfH_1 $ and $ \forall \,  (v,w) \in \bfH_1 \colon \bfA( v , w ) = ( w , A v ) $, let $ \mu \colon \bbH \to \R $ be a function which satisfies \smash{$ \sum_{ h \in \bbH } \frac{ \abs{ \mu_h }^2 }{ \abs{ \lambda_h } } < \infty $}\vspace{0.5mm}, let $ \bfB \in \mathrm{HS}( H, \bfH_0 ) $ be the linear operator which satisfies for all $ v \in H $ that $ \bfB v = \bigl( 0 , \sum_{ h \in \bbH } \mu_h \langle h , v \rangle_H h \bigr) $, and let $ \bfX^I = ( X^{ I, 1 }, X^{ I, 2 } ) \colon \Omega \to \bfP_I ( \bfH_0 ) $, $ I \in \cP( \bbH ) $, be random variables which satisfy for all $ I \in \cP( \bbH ) $ that it holds $ \P $-a.s.\ that \smash{$ \bfX^I = \int_0^T \e^{ ( T - s )\bfA } \bfP_I \bfB \dWs $}.
\subsection{Lower bounds for the squared norm}
\begin{lemma} \label{lem:solution}
Assume the setting in Section~\ref{subsec:setting_lower}. Then for all $ I \in \cP( \bbH ) $ it holds $ \P $-a.s.\ that
\begin{equation}
\bfX^I = \bfP_I \bfX^\bbH = \twovector{ X^{ I, 1 } }{ X^{ I, 2 } } = %
\twovector{ \sum_{ h \in I } \Bigl( \frac{ \mu_h }{ \abs{ \lambda_h }^{ \nicefrac{1}{2} } }  \int_0^T \sin \bigl( \abs{ \lambda_h }^{ \nicefrac{1}{2} }(T-s) \bigr) \ud \langle h , W_s \rangle_H \Bigr) h }%
{ \sum_{ h \in I } \Bigl( \frac{ \mu_h }{ \abs{ \lambda_h }^{ \nicefrac{1}{2} } } \int_0^T \cos \bigl( \abs{ \lambda_h }^{ \nicefrac{1}{2} }(T-s) \bigr) \ud \langle h , W_s \rangle_H \Bigr) \abs{ \lambda_h }^{ \nicefrac{1}{2} } h }.
\end{equation}
\end{lemma}
\begin{proof}[Proof of Lemma~\ref{lem:solution}]
Lemma~2.5 in~\cite{DeNauroisJentzenWelti2015ArXiv}
proves that it holds $ \P $-a.s.\ that
\begin{equation} \label{eq:solution,1}
\begin{split}
 \bfX^{ \bbH } & = \int_0^T \e^{ (T-s)\bfA } \bfB \dWs = \sum_{ h \in \bbH } \int_0^T \e^{ (T-s)\bfA } \bfB h \ud \langle h , W_s \rangle_H \\
& = \sum_{ h \in \bbH } \twovector{ \mu_h \int_0^T ( -A )^{ - \nicefrac{1}{2} } \sin \bigl( ( -A )^{ \nicefrac{1}{2} }(T-s) \bigr) h \ud \langle h , W_s \rangle_H }%
{ \mu_h \int_0^T \cos \bigl( ( -A )^{ \nicefrac{1}{2} }(T-s) \bigr) h \ud \langle h , W_s \rangle_H } \\
& = \sum_{ h \in \bbH } \twovector{ \frac{ \mu_h }{ \abs{ \lambda_h }^{ \nicefrac{1}{2} } } \int_0^T \sin \bigl( \abs{ \lambda_h }^{ \nicefrac{1}{2} }(T-s) \bigr) h \ud \langle h , W_s \rangle_H }%
{ \mu_h \int_0^T \cos \bigl( \abs{ \lambda_h }^{ \nicefrac{1}{2} }(T-s) \bigr) h \ud \langle h , W_s \rangle_H } \\
& = \twovector{ \sum_{ h \in \bbH } \Bigl( \frac{ \mu_h }{ \abs{ \lambda_h }^{ \nicefrac{1}{2} } }  \int_0^T \sin \bigl( \abs{ \lambda_h }^{ \nicefrac{1}{2} }(T-s) \bigr) \ud \langle h , W_s \rangle_H \Bigr) h }%
{ \sum_{ h \in \bbH } \Bigl( \frac{ \mu_h }{ \abs{ \lambda_h }^{ \nicefrac{1}{2} } } \int_0^T \cos \bigl( \abs{ \lambda_h }^{ \nicefrac{1}{2} }(T-s) \bigr) \ud \langle h , W_s \rangle_H \Bigr) \abs{ \lambda_h }^{ \nicefrac{1}{2} } h }.
\end{split}
\end{equation}
Furthermore,
Lemma~2.7 in~\cite{DeNauroisJentzenWelti2015ArXiv}
shows for all $ I \in \cP( \bbH ) $ that it holds $ \P $-a.s.\ that
\begin{equation}
\bfP_I \bfX^{ \bbH } = \int_0^T \bfP_I \e^{ (T-s) \bfA }  \bfB \dWs = \int_0^T \e^{ (T-s)\bfA } \bfP_I \bfB \dWs = \bfX^I.
\end{equation} 
This and \eqref{eq:solution,1} complete the proof of Lemma \ref{lem:solution}.
\end{proof}
\begin{lemma} \label{lem:gaussian}
Assume the setting in Section~\ref{subsec:setting_lower} and let $ I \in \cP( \bbH ) $. Then
\begin{enumerate}[(i)]
\item \label{lem:gaussian:i1}
it holds that $ \langle h, X^{ I, 1 } \rangle_{ H_0 } $,  $ h \in \bbH $, is a family of independent centred Gaussian random variables,
\item \label{lem:gaussian:i2}
it holds that \smash{$ \bigl\langle \abs{ \lambda_h }^{ \nicefrac{1}{2} } h, X^{ I, 2 } \bigr\rangle_{ H_{ - \nicefrac{1}{2} } } $}, $ h \in \bbH $, is a family of independent centred Gaussian random variables, and
\item \label{lem:gaussian:i3}
it holds for all $ h \in \bbH $ that
\begin{align}
\Var \bigl( \langle h, X^{ I, 1 } \rangle_{ H_0 } \bigr) & = \bbone_I (h) \frac{ \abs{ \mu_h }^2}{ \abs{ \lambda_h } } \frac{1}{2} \biggl( T - \frac{ \sin \bigl( 2 \abs{ \lambda_h }^{ \nicefrac{1}{2} } T \bigr) }{ 2 \abs{ \lambda_h }^{ \nicefrac{1}{2} } } \biggr), \\
\Var \Bigl( \bigl\langle \abs{ \lambda_h }^{ \nicefrac{1}{2} } h, X^{ I, 2 } \bigr\rangle_{ H_{ - \nicefrac{1}{2} } } \Bigr) & = \bbone_I (h) \frac{ \abs{ \mu_h }^2}{ \abs{ \lambda_h } } \frac{1}{2} \biggl( T + \frac{ \sin \bigl( 2 \abs{ \lambda_h }^{ \nicefrac{1}{2} } T \bigr) }{ 2 \abs{ \lambda_h }^{ \nicefrac{1}{2} } } \biggr), \\
\Cov \Bigl( \langle h, X^{ I, 1 } \rangle_{ H_0 }, \bigl\langle \abs{ \lambda_h }^{ \nicefrac{1}{2} } h, X^{ I, 2 } \bigr\rangle_{ H_{ - \nicefrac{1}{2} } } \Bigr) & = \bbone_I (h) \frac{ \abs{ \mu_h }^2}{ \abs{ \lambda_h } } \biggl( \frac{ 1 - \cos \bigl( 2 \abs{ \lambda_h }^{ \nicefrac{1}{2} } T \bigr) }{ 4 \abs{ \lambda_h }^{ \nicefrac{1}{2} } } \biggr).
\end{align}
\end{enumerate}
\end{lemma}
\begin{proof}[Proof of Lemma~\ref{lem:gaussian}]
Observe that Lemma~\ref{lem:solution} implies \eqref{lem:gaussian:i1} and \eqref{lem:gaussian:i2}. It thus remains to prove \eqref{lem:gaussian:i3}. Lemma~\ref{lem:solution} implies for all $ h \in \bbH $ that it holds $ \P $-a.s.\ that
\begin{align}
\langle h, X^{ I, 1 } \rangle_{ H_0 } & = \bbone_I (h) \frac{ \mu_h }{ \abs{ \lambda_h }^{ \nicefrac{1}{2} } } \int_0^T \sin \bigl( \abs{ \lambda_h }^{ \nicefrac{1}{2} }(T-s) \bigr) \ud \langle h , W_s \rangle_H, \\
\bigl\langle \abs{ \lambda_h }^{ \nicefrac{1}{2} } h, X^{ I, 2 } \bigr\rangle_{ H_{ - \nicefrac{1}{2} } } & = \bbone_I (h) \frac{ \mu_h }{ \abs{ \lambda_h }^{ \nicefrac{1}{2} } } \int_0^T \cos \bigl( \abs{ \lambda_h }^{ \nicefrac{1}{2} }(T-s) \bigr) \ud \langle h , W_s \rangle_H.
\end{align}
It\^{o}'s isometry hence shows for all $ h \in \bbH $ that
\begin{align}
\begin{split}
\Var \bigl( \langle h, X^{ I, 1 } \rangle_{ H_0 } \bigr) & = \E \bigl[ \abs{ \langle h, X^{ I, 1 } \rangle_{ H_0 } }^2 \bigr] = \bbone_I (h) \frac{ \abs{ \mu_h }^2 }{ \abs{ \lambda_h } } \int_0^T \abs[\big]{ \sin \bigl( \abs{ \lambda_h }^{ \nicefrac{1}{2} }(T-s) \bigr) }^2 \ds \\
& = \bbone_I (h) \frac{ \abs{ \mu_h }^2}{ \abs{ \lambda_h } } \frac{1}{2} \biggl( T - \frac{ \sin \bigl( 2 \abs{ \lambda_h }^{ \nicefrac{1}{2} } T \bigr) }{ 2 \abs{ \lambda_h }^{ \nicefrac{1}{2} } } \biggr),
\end{split} \\
\begin{split}
\Var \Bigl( \bigl\langle \abs{ \lambda_h }^{ \nicefrac{1}{2} } h, X^{ I, 2 } \bigr\rangle_{ H_{ - \nicefrac{1}{2} } } \Bigr) & = \E \Bigl[ \abs[\Big]{ \bigl\langle \abs{ \lambda_h }^{ \nicefrac{1}{2} } h, X^{ I, 2 } \bigr\rangle_{ H_{ - \nicefrac{1}{2} } } }^2 \Bigr] \\
& = \bbone_I (h) \frac{ \abs{ \mu_h }^2 }{ \abs{ \lambda_h } } \int_0^T \abs[\big]{ \cos \bigl( \abs{ \lambda_h }^{ \nicefrac{1}{2} }(T-s) \bigr) }^2 \ds \\
& = \bbone_I (h) \frac{ \abs{ \mu_h }^2}{ \abs{ \lambda_h } } \frac{1}{2} \biggl( T + \frac{ \sin \bigl( 2 \abs{ \lambda_h }^{ \nicefrac{1}{2} } T \bigr) }{ 2 \abs{ \lambda_h }^{ \nicefrac{1}{2} } } \biggr).
\end{split}
\end{align}
Furthermore, observe for all $ h \in \bbH $ that
\begin{equation}
\begin{split}
& \Cov \Bigl( \langle h, X^{ I, 1 } \rangle_{ H_0 }, \bigl\langle \abs{ \lambda_h }^{ \nicefrac{1}{2} } h, X^{ I, 2 } \bigr\rangle_{ H_{ - \nicefrac{1}{2} } } \Bigr) = \E \Bigl[ \langle h, X^{ I, 1 } \rangle_{ H_0 } \bigl\langle \abs{ \lambda_h }^{ \nicefrac{1}{2} } h, X^{ I, 2 } \bigr\rangle_{ H_{ - \nicefrac{1}{2} } } \Bigr] \\
& = \bbone_I (h) \frac{ \abs{ \mu_h }^2 }{ \abs{ \lambda_h } } \int_0^T \sin \bigl( \abs{ \lambda_h }^{ \nicefrac{1}{2} }(T-s) \bigr) \cos \bigl( \abs{ \lambda_h }^{ \nicefrac{1}{2} }(T-s) \bigr) \ds \\
& = \bbone_I (h) \frac{ \abs{ \mu_h }^2}{ \abs{ \lambda_h } } \biggl( \frac{ \abs[\big]{ \sin \bigl( \abs{ \lambda_h }^{ \nicefrac{1}{2} } T \bigr) }^2 }{ 2 \abs{ \lambda_h }^{ \nicefrac{1}{2} } } \biggr) \\
& = \bbone_I (h) \frac{ \abs{ \mu_h }^2}{ \abs{ \lambda_h } } \biggl( \frac{ 1 - \cos \bigl( 2 \abs{ \lambda_h }^{ \nicefrac{1}{2} } T \bigr) }{ 4 \abs{ \lambda_h }^{ \nicefrac{1}{2} } } \biggr).
\end{split}
\end{equation}
The proof of Lemma~\ref{lem:gaussian} is thus completed.
\end{proof}
\begin{cor} \label{cor:covop}
Assume the setting in Section~\ref{subsec:setting_lower} and let $ I \in \cP( \bbH ) $. Then it holds for all $ (v,w) \in \bfP_I ( \bfH_0 ) $ that
\begin{align*}
\CovOp ( \bfX^I ) ( v, w ) & = \frac{1}{2} \frac{ \abs{ \mu_h }^2}{ \abs{ \lambda_h } } \sum_{ h \in I } \biggl[ \biggl( T - \frac{ \sin \bigl( 2 \abs{ \lambda_h }^{ \nicefrac{1}{2} } T \bigr) }{ 2 \abs{ \lambda_h }^{ \nicefrac{1}{2} } } \biggr) \langle h, v \rangle_{ H_0 } \twovector{ h }{ 0 } \\
& \hphantom{ = \frac{1}{2} \frac{ \abs{ \mu_h }^2}{ \abs{ \lambda_h } } \sum_{ h \in \bbH } \biggl[} + \biggl( \frac{ 1 - \cos \bigl( 2 \abs{ \lambda_h }^{ \nicefrac{1}{2} } T \bigr) }{ 2 \abs{ \lambda_h }^{ \nicefrac{1}{2} } } \biggr) \bigl\langle \abs{ \lambda_h }^{ \nicefrac{1}{2} } h, w \bigr\rangle_{ H_{ - \nicefrac{1}{2} } } \twovector{ h }{ 0 } \yesnumber \\
& \hphantom{ = \frac{1}{2} \frac{ \abs{ \mu_h }^2}{ \abs{ \lambda_h } } \sum_{ h \in \bbH } \biggl[} + \biggl( \frac{ 1 - \cos \bigl( 2 \abs{ \lambda_h }^{ \nicefrac{1}{2} } T \bigr) }{ 2 \abs{ \lambda_h }^{ \nicefrac{1}{2} } } \biggr) \langle h, v \rangle_{ H_0 }  \twovector{ 0 }{ \abs{ \lambda_h }^{ \nicefrac{1}{2} } h } \\
& \hphantom{ = \frac{1}{2} \frac{ \abs{ \mu_h }^2}{ \abs{ \lambda_h } } \sum_{ h \in \bbH } \biggl[} + \biggl( T + \frac{ \sin \bigl( 2 \abs{ \lambda_h }^{ \nicefrac{1}{2} } T \bigr) }{ 2 \abs{ \lambda_h }^{ \nicefrac{1}{2} } } \biggr) \bigl\langle \abs{ \lambda_h }^{ \nicefrac{1}{2} } h, w \bigr\rangle_{ H_{ - \nicefrac{1}{2} } } \twovector{ 0 }{ \abs{ \lambda_h }^{ \nicefrac{1}{2} } h } \biggr]
\in \bfP_I ( \bfH_0 ).
\end{align*}
\end{cor}
\begin{proof}[Proof of Corollary~\ref{cor:covop}]
Lemma~\ref{lem:gaussian} and Lemma~\ref{lem:solution} prove for all $ x_1 = ( v_1, w_1) $, $ x_2 = ( v_2, w_2 ) \in \bfP_I ( \bfH_0 ) $ that
\begin{equation}
\begin{split}
&  \langle x_1, \CovOp ( \bfX^I ) x_2 \rangle_{ \bfH_0 } = \Cov \bigl( \langle x_1, \bfX^I \rangle_{ \bfH_0 }, \langle x_2, \bfX^I \rangle_{ \bfH_0 } \bigr) = \E \bigl[ \langle x_1, \bfX^I \rangle_{ \bfH_0 } \langle x_2, \bfX^I \rangle_{ \bfH_0 } \bigr] \\
& = \E \bigl[ \bigl( \langle v_1, X^{ I, 1 } \rangle_{ H_0 } + \langle w_1, X^{ I, 2 } \rangle_{ H_{ - \nicefrac{1}{2} } } \bigl) \bigl( \langle v_2, X^{ I, 1 } \rangle_{ H_0 } + \langle w_2, X^{ I, 2 } \rangle_{ H_{ - \nicefrac{1}{2} } } \bigl) \bigr] \\
& = \sum_{ h \in \bbH } \Var \bigl( \langle h, X^{ I, 1 } \rangle_{ H_0 } \bigr) \langle h, v_1 \rangle_{ H_0 } \langle h, v_2 \rangle_{ H_0 } \\
& \quad + \sum_{ h \in \bbH } \Cov \Bigl( \langle h, X^{ I, 1 } \rangle_{ H_0 }, \bigl\langle \abs{ \lambda_h }^{ \nicefrac{1}{2} } h, X^{ I, 2 } \bigr\rangle_{ H_{ - \nicefrac{1}{2} } } \Bigr) \langle h, v_1 \rangle_{ H_0 } \bigl\langle \abs{ \lambda_h }^{ \nicefrac{1}{2} } h, w_2 \bigr\rangle_{ H_{ - \nicefrac{1}{2} } } \\
& \quad + \sum_{ h \in \bbH } \Cov \Bigl( \langle h, X^{ I, 1 } \rangle_{ H_0 }, \bigl\langle \abs{ \lambda_h }^{ \nicefrac{1}{2} } h, X^{ I, 2 } \bigr\rangle_{ H_{ - \nicefrac{1}{2} } } \Bigr) \langle h, v_2 \rangle_{ H_0 } \bigl\langle \abs{ \lambda_h }^{ \nicefrac{1}{2} } h, w_1 \bigr\rangle_{ H_{ - \nicefrac{1}{2} } } \\
& \quad + \sum_{ h \in \bbH } \Var \Bigl( \bigl\langle \abs{ \lambda_h }^{ \nicefrac{1}{2} } h, X^{ I, 2 } \bigr\rangle_{ H_{ - \nicefrac{1}{2} } } \Bigr) \bigl\langle \abs{ \lambda_h }^{ \nicefrac{1}{2} } h, w_1 \bigr\rangle_{ H_{ - \nicefrac{1}{2} } } \bigl\langle \abs{ \lambda_h }^{ \nicefrac{1}{2} } h, w_2 \bigr\rangle_{ H_{ - \nicefrac{1}{2} } } \\
& = \biggl\langle v_1, \sum_{ h \in \bbH } \Bigl[ \Var \bigl( \langle h, X^{ I, 1 } \rangle_{ H_0 } \bigr) \langle h, v_2 \rangle_{ H_0 } \\
& \qquad\qquad\quad\ + \Cov \Bigl( \langle h, X^{ I, 1 } \rangle_{ H_0 }, \bigl\langle \abs{ \lambda_h }^{ \nicefrac{1}{2} } h, X^{ I, 2 } \bigr\rangle_{ H_{ - \nicefrac{1}{2} } } \Bigr) \bigl\langle \abs{ \lambda_h }^{ \nicefrac{1}{2} } h, w_2 \bigr\rangle_{ H_{ - \nicefrac{1}{2} } } \Bigr] h \biggl\rangle_{ H_0 } \\
& \quad + \biggl\langle w_1, \sum_{ h \in \bbH } \Bigl[ \Cov \Bigl( \langle h, X^{ I, 1 } \rangle_{ H_0 }, \bigl\langle \abs{ \lambda_h }^{ \nicefrac{1}{2} } h, X^{ I, 2 } \bigr\rangle_{ H_{ - \nicefrac{1}{2} } } \Bigr) \langle h, v_2 \rangle_{ H_0 } \\
& \qquad\qquad\qquad \ + \Var \Bigl( \bigl\langle \abs{ \lambda_h }^{ \nicefrac{1}{2} } h, X^{ I, 2 } \bigr\rangle_{ H_{ - \nicefrac{1}{2} } } \Bigr) \bigl\langle \abs{ \lambda_h }^{ \nicefrac{1}{2} } h, w_2 \bigr\rangle_{ H_{ - \nicefrac{1}{2} } } \Bigr] \abs{ \lambda_h }^{ \nicefrac{1}{2} } h \biggr\rangle_{ H_{ - \nicefrac{1}{2} } } \\
& = \biggl\langle x_1, \sum_{ h \in \bbH } \Bigl[ \Var \bigl( \langle h, X^{ I, 1 } \rangle_{ H_0 } \bigr) \langle h, v_2 \rangle_{ H_0 } \\
& \qquad\qquad\quad\ + \Cov \Bigl( \langle h, X^{ I, 1 } \rangle_{ H_0 }, \bigl\langle \abs{ \lambda_h }^{ \nicefrac{1}{2} } h, X^{ I, 2 } \bigr\rangle_{ H_{ - \nicefrac{1}{2} } } \Bigr) \bigl\langle \abs{ \lambda_h }^{ \nicefrac{1}{2} } h, w_2 \bigr\rangle_{ H_{ - \nicefrac{1}{2} } } \Bigr] \twovector{ h }{ 0 } \biggr\rangle_{ \bfH_0 } \\
& \quad + \biggl\langle x_1, \sum_{ h \in \bbH } \Bigl[ \Cov \Bigl( \langle h, X^{ I, 1 } \rangle_{ H_0 }, \bigl\langle \abs{ \lambda_h }^{ \nicefrac{1}{2} } h, X^{ I, 2 } \bigr\rangle_{ H_{ - \nicefrac{1}{2} } } \Bigr) \langle h, v_2 \rangle_{ H_0 } \\
& \qquad\qquad\qquad \ + \Var \Bigl( \bigl\langle \abs{ \lambda_h }^{ \nicefrac{1}{2} } h, X^{ I, 2 } \bigr\rangle_{ H_{ - \nicefrac{1}{2} } } \Bigr) \bigl\langle \abs{ \lambda_h }^{ \nicefrac{1}{2} } h, w_2 \bigr\rangle_{ H_{ - \nicefrac{1}{2} } } \Bigr] \twovector{ 0 }{ \abs{ \lambda_h }^{ \nicefrac{1}{2} } h } \biggr\rangle_{ \bfH_0 }.
\end{split}
\end{equation}
This and again Lemma~\ref{lem:gaussian} complete the proof of Corollary~\ref{cor:covop}.
\end{proof}
\begin{lemma} \label{lem:L2norm}
Assume the setting in Section~\ref{subsec:setting_lower} and let $ I \in \cP( \bbH ) $. Then it holds for all $ i \in \{ 1, 2 \} $ that $ \bfX^I \in \mathcal{L}^2( \P; \bfH_0 ) $ and
\begin{align}
\E \bigl[ \norm{ \bfX^I }_{ \bfH_0 }^2 \bigr] & = T \sum_{ h \in I } \frac{ \abs{ \mu_h }^2 }{ \abs{ \lambda_h } } < \infty, \\
\E \Bigl[ \norm{ X^{ I, i } }_{ H_{ \nicefrac{1}{2} - \nicefrac{i}{2} } }^2 \Bigr] & = \frac{1}{2} \sum_{ h \in I } \frac{ \abs{ \mu_h }^2 }{ \abs{ \lambda_h } } \biggl( T + \frac{ \sin \bigl( 2 \abs{ \lambda_h }^{ \nicefrac{1}{2} } T \bigr) }{ (-1)^i \+ 2 \abs{ \lambda_h }^{ \nicefrac{1}{2} } } \biggr) < \infty.
\end{align}
\end{lemma}
\begin{proof}[Proof of Lemma~\ref{lem:L2norm}]
It\^{o}'s isometry and
Lemma~2.6 in~\cite{DeNauroisJentzenWelti2015ArXiv}
imply that
\begin{equation}
\E \bigl[ \norm{ \bfX^I }_{ \bfH_0 }^2 \bigr] = \E \biggl[ \norm[\bigg]{ \int_0^T \e^{ ( T - s ) \bfA } \bfP_I \bfB \dWs }_{ \bfH_0 }^2 \biggr] = T \norm{ \bfP_I \bfB }_{ \mathrm{HS}( H, \bfH_0 ) }^2 =  T \sum_{ h \in I } \frac{  \abs{ \mu_h }^2 }{ \abs{ \lambda_h } } < \infty.
\end{equation}
In addition, Lemma~\ref{lem:gaussian} shows for all $ i \in \{ 1, 2 \} $ that
\begin{equation}
\begin{split}
\E \Bigl[ \norm{ X^{ I, i } }_{ H_{ \nicefrac{1}{2} - \nicefrac{i}{2} } }^2 \Bigr] & = \sum_{ h \in \bbH } \E \Bigl[ \abs[\Big]{ \bigl\langle \abs{ \lambda_h }^{ \nicefrac{i}{2} - \nicefrac{1}{2} } h, X^{ I, i } \bigr\rangle_{ H_{ \nicefrac{1}{2} - \nicefrac{i}{2} } } }^2 \Bigr] \\
& = \frac{1}{2} \sum_{ h \in I } \frac{ \abs{ \mu_h }^2 }{ \abs{ \lambda_h } } \biggl( T + \frac{ \sin \bigl( 2 \abs{ \lambda_h }^{ \nicefrac{1}{2} } T \bigr) }{ (-1)^i \+ 2 \abs{ \lambda_h }^{ \nicefrac{1}{2} } } \biggr) < \infty.
\end{split}
\end{equation}
The proof of Lemma~\ref{lem:L2norm} is thus completed.
\end{proof}
\begin{prop} \label{prop:lower_bound}
Assume the setting in Section~\ref{subsec:setting_lower} and assume $ \inf _{ h \in \bbH } \abs{ \mu_h } > 0 $. Then it holds for all $ I \in \cP( \bbH ) \setminus \{ \bbH \} $ that
\begin{equation}
\E \bigl[ \norm{ \bfX^\bbH }_{ \bfH_0 }^2 \bigr] - \E \bigl[ \norm{ \bfX^I }_{ \bfH_0 }^2 \bigr] = \E \bigl[ \norm{ \bfX^{ \bbH \setminus I } }_{ \bfH_0 }^2 \bigr] \geq T \inf_{ h \in \bbH } \abs{ \mu_{ h } }^2 \sum_{ h \in \bbH \setminus I } \frac{1}{ \abs{ \lambda_h } } > 0.
\end{equation}
\end{prop}
\begin{proof}[Proof of~Proposition \ref{prop:lower_bound}] 
Orthogonality and Lemma~\ref{lem:solution} imply for all $ I \in \cP( \bbH ) \setminus \{ \bbH \} $ that
\begin{equation}
\begin{split}
\E \bigl[ \norm{ \bfX^I }_{ \bfH_0 }^2 \bigr] + \E \bigl[ \norm{ \bfX^{ \bbH \setminus I } }_{ \bfH_0 }^2 \bigr] & = \E \bigl[ \norm{ \bfP_I \bfX^\bbH }_{ \bfH_0 }^2 \bigr] + \E \bigl[ \norm{ \bfP_{ \bbH \setminus I } \bfX^\bbH }_{ \bfH_0 }^2 \bigr] \\
& = \E \bigl[ \norm{ ( \bfP_I + \bfP_{ \bbH \setminus I } ) \bfX^\bbH }_{ \bfH_0 }^2 \bigr] = \E \bigl[ \norm{ \bfX^\bbH }_{ \bfH_0 }^2 \bigr].
\end{split}
\end{equation}
This and Lemma~\ref{lem:L2norm} show for all $ I \in \cP( \bbH ) \setminus \{ \bbH \} $ that
\begin{equation}
\begin{split}
\E \bigl[ \norm{ \bfX^\bbH }_{ \bfH_0 }^2 \bigr] - \E \bigl[ \norm{ \bfX^I }_{ \bfH_0 }^2 \bigr] & = \E \bigl[ \norm{ \bfX^{ \bbH \setminus I } }_{ \bfH_0 }^2 \bigr] \\
& = T \sum_{ h \in \bbH \setminus I } \frac{  \abs{ \mu_h }^2 }{ \abs{ \lambda_h } } \geq T \inf_{ h \in \bbH } \abs{ \mu_{ h } }^2 \sum_{ h \in \bbH \setminus I } \frac{1}{ \abs{ \lambda_h } } > 0.
\end{split}
\end{equation}
The proof of Proposition~\ref{prop:lower_bound} is thus completed.
\end{proof}
In Corollary~\ref{cor:lower_bound_inf} and Corollary~\ref{cor:lower_bound_delta} below lower bounds on the weak approximation error with the squared norm as test function are presented. Our proofs of Corollary~\ref{cor:lower_bound_inf} and Corollary~\ref{cor:lower_bound_delta} use the following elementary and well-known lemma (cf., e.g., Proposition~6.4 in Conus et al.~\cite{ConusJentzenKurniawan2014}).
\begin{lemma} \label{lem:sum}
Let $ p \in ( 0 , \infty ) $, $ \delta \in ( - \infty, \nicefrac{1}{2} - \nicefrac{1}{ (2 p ) } ) $. Then it holds for all $ N \in \N $ that
\begin{equation}
\sum_{ n = N + 1 }^{ \infty } n^{ p ( 2 \delta - 1 ) } \geq \frac{ N^{ p ( 2 \delta - 1 ) + 1 } }{ [ p ( 1 - 2 \delta ) - 1 ] 2^{ p ( 1 - 2 \delta ) - 1 } }.
\end{equation}
\end{lemma}
\begin{proof}[Proof of Lemma~\ref{lem:sum}]
Observe that the assumption that $ \delta \in ( - \infty, \nicefrac{1}{2} - \nicefrac{1}{ (2 p ) } ) $ ensures that $ p ( 2 \delta - 1 ) \in ( - \infty, - 1 ) $. This implies for all $ N \in \N $ that
\begin{equation}
\begin{split}
 \sum_{ n = N + 1 }^{ \infty } n^{ p ( 2 \delta - 1 ) } &=  \sum_{ n = N + 1 }^{ \infty } \int_{ n }^{ n+1 } n^{ p ( 2 \delta - 1 ) } \dx \geq \sum_{ n = N+1 }^{ \infty } \int_{ n }^{ n+1 } x^{ p ( 2 \delta - 1 ) } \dx \\
& = \int_{ N+1 }^{ \infty } x^{ p ( 2 \delta - 1 ) } \dx = - \frac{ ( N + 1 )^{ p ( 2 \delta - 1 ) + 1 } }{ p ( 2 \delta - 1 ) + 1 } \\
& \geq \frac{ N^{ p ( 2 \delta - 1 ) + 1 } }{ [ p ( 1 - 2 \delta ) - 1 ] 2^{ p ( 1 - 2 \delta ) - 1 } }.
\end{split}
\end{equation}
This completes the proof of Lemma~\ref{lem:sum}.
\end{proof}
\begin{cor} \label{cor:lower_bound_inf}
Assume the setting in Section~\ref{subsec:setting_lower}, let $ c \in ( 0 , \infty ) , p \in ( 1 , \infty ) $, let $ e \colon \N \to \bbH $ be a bijection which satisfies for all $ n \in \N $ that $ \lambda_{ e_n } = - c n^p $, and let $ I_N \in \cP( \bbH ) $, $ N \in \N $, be the sets which satisfy for all $ N \in \N $ that $ I_N = \{ e_1, e_2, \ldots, e_N \} \subseteq \bbH $. Then it holds for all $ N \in \N $ that
\begin{equation}
\E \bigl[ \norm{ \bfX^\bbH }^2_{ \bfH_0 } \bigr] - \E \bigl[ \norm{ \bfX^{ I_N } }^2_{ \bfH_0 } \bigr] \geq \frac{ T \inf_{ h \in \bbH } \abs{ \mu_h }^2 N^{ 1 - p } }{ c \+ ( p - 1) 2^{ p - 1 } }.
\end{equation}
\end{cor}
\begin{proof}[Proof of Corollary~\ref{cor:lower_bound_inf}]
Proposition~\ref{prop:lower_bound} and Lemma~\ref{lem:sum} prove for all $ N \in \N $ that
\begin{equation}
\begin{split}
\E \bigl[ \norm{ \bfX^\bbH }_{ \bfH_0 }^2 \bigr] - \E \bigl[ \norm{ \bfX^{ I_N } }_{ \bfH_0 }^2 \bigr] & \geq T \inf_{ h \in \bbH } \abs{ \mu_{ h } }^2 \sum_{ h \in \bbH \setminus I_N } \frac{1}{ \abs{ \lambda_h } } = c^{ - 1 } \tp T \inf_{ h \in \bbH } \abs{ \mu_{ h } }^2 \sum_{ n = N+1 }^\infty \frac{1}{ n^p } \\
& \geq \frac{ T \inf_{ h \in \bbH } \abs{ \mu_h }^2  N^{ 1 - p } }{ c \+ ( p - 1) 2^{ p - 1 } }.
\end{split}
\end{equation}
The proof of Corollary~\ref{cor:lower_bound_inf} is thus completed.
\end{proof}
\begin{cor} \label{cor:lower_bound_delta}
Assume the setting in Section~\ref{subsec:setting_lower}, let $ c, p \in ( 0 , \infty ) $, $ \delta \in ( - \infty, \nicefrac{1}{2} - \nicefrac{1}{ (2 p ) } ) $, let $ e \colon \N \to \bbH $ be a bijection which satisfies for all $ n \in \N $ that $ \lambda_{ e_n } = -c n^p $, let $ I_N \in \cP( \bbH ) $, $ N \in \N $, be the sets which satisfy for all $ N \in \N $ that $ I_N = \{ e_1, e_2, \ldots, e_N \} \subseteq \bbH $, and assume for all $ h \in \bbH $ that $ \abs{ \mu_h } = \abs{ \lambda_{ h } }^\delta $. Then it holds for all $ N \in \N $ that
\begin{equation}
\E \bigl[ \norm{ \bfX^\bbH }^2_{ \bfH_0 } \bigr] - \E \bigl[ \norm{ \bfX^{ I_N } }^2_{ \bfH_0 } \bigr] \geq \frac{ T c^{ 2 \delta - 1 } N^{ p ( 2 \delta - 1 ) + 1 } }{ [ p ( 1 - 2 \delta ) - 1 ] 2^{ p ( 1 - 2 \delta ) - 1 } }.
\end{equation}
\end{cor}
\begin{proof}[Proof of Corollary~\ref{cor:lower_bound_delta}]
Proposition~\ref{prop:lower_bound}, Lemma~\ref{lem:L2norm}, and Lemma~\ref{lem:sum} show for all $ N \in \N $ that
\begin{equation}
\begin{split}
\E \bigl[ \norm{ \bfX^\bbH }^2_{ \bfH_0 } \bigr] - \E \bigl[ \norm{ \bfX^{ I_N } }^2_{ \bfH_0 } \bigr] & = T \sum_{ h \in \bbH \setminus I_N } \frac{ \abs{ \mu_h }^2 }{ \abs{ \lambda_h } } = T \sum_{ h \in \bbH \setminus I_N } \abs{ \lambda_h }^{ 2 \delta - 1 } \\
& = T c^{ 2 \delta - 1 } \sum_{ n = N+1 }^\infty n^{ p ( 2 \delta - 1 ) } \geq \frac{ T c^{ 2 \delta - 1 } N^{ p ( 2 \delta - 1 ) + 1 } }{ [ p ( 1 - 2 \delta ) - 1 ] 2^{ p ( 1 - 2 \delta ) - 1 } }.
\end{split}
\end{equation}
This completes the proof of Corollary~\ref{cor:lower_bound_delta}.
\end{proof}
\subsection{Lower bounds for the weak error of a particular regular test function} \label{subsec:lower_bounds_regular}
The next proposition, Proposition~\ref{prop:lower_bound_exp} below, follows directly from Lemma~\ref{lem:gaussian} above and Lemma~9.5 in Jentzen \& Kurniawan~\cite{JentzenKurniawan2015}.
\begin{prop} \label{prop:lower_bound_exp}
Assume the setting in Section \ref{subsec:setting_lower} and let $ \varphi_i \colon \bfH_0 \to \R $, $ i \in \{ 1, 2 \} $, be the functions which satisfy for all $ i \in \{ 1, 2 \} $, $ ( v_1, v_2 ) \in \bfH_0 $ that \smash{$ \varphi_i( v_1 , v_2 ) = \exp \bigl( - \norm{ v_i }^2_{ H_{ \nicefrac{1}{2} - \nicefrac{i}{2} } } \bigr) $}. Then it holds for all $ i \in \{ 1, 2 \} $, $ I \in \cP( \bbH ) $ that $ \varphi_i \in C_{ \mathrm{b} }^2( \bfH_0, \R ) $ and
\begin{equation}
\E [ \varphi_i( \bfX^I ) ] - \E[ \varphi_i( \bfX^\bbH ) ] \geq \frac{ \E \Bigl[ \norm{ X^{ \bbH, i } }_{ H_{ \nicefrac{1}{2} - \nicefrac{i}{2} } }^2 \Bigr] - \E \Bigl[ \norm{ X^{ I, i } }_{ H_{ \nicefrac{1}{2} - \nicefrac{i}{2} } }^2 \Bigr] }{ \exp \Bigl( 6 \E \Bigl[ \norm{ X^{ \bbH, i } }_{ H_{ \nicefrac{1}{2} - \nicefrac{i}{2} } }^2 \Bigr] \Bigr) }.
\end{equation}
\end{prop}
\begin{cor} \label{cor:lower_bound_exp}
Assume the setting in Section~\ref{subsec:setting_lower}, let $ c, p \in ( 0 , \infty ) $, $ \delta \in ( - \infty, \nicefrac{1}{2} - \nicefrac{1}{ (2 p ) } ) $, let $ e \colon \N \to \bbH $ be a bijection which satisfies for all $ n \in \N $ that $ \lambda_{ e_n } = -c n^p $, let $ I_N \in \cP( \bbH ) $, $ N \in \N $, be the sets which satisfy for all $ N \in \N $ that $ I_N = \{ e_1, e_2, \ldots, e_N \} \subseteq \bbH $, assume for all $ h \in \bbH $ that $  \abs{ \mu_h } = \abs{ \lambda_{ h } }^\delta $, and let $ \varphi_i \colon \bfH_0 \to \R $, $ i \in \{ 1, 2 \} $, be the functions which satisfy for all $ i \in \{ 1, 2 \} $, $ ( v_1, v_2 ) \in \bfH_0 $ that \smash{$ \varphi_i( v_1 , v_2 ) = \exp \bigl( - \norm{ v_i }^2_{ H_{ \nicefrac{1}{2} - \nicefrac{i}{2} } } \bigr) $}. Then it holds for all $ i \in \{ 1, 2 \} $, $ N \in \N $ that $ \varphi_i \in C_{ \mathrm{b} }^2( \bfH_0, \R ) $ and
\begin{equation}
\E [ \varphi_i( \bfX^{ I_N } ) ] - \E[ \varphi_i( \bfX^\bbH ) ] \geq \biggl[ 1 + \inf_{ x \in [ 2 c^{\nicefrac{1}{2} } T, \infty ) } \frac{ \sin( x ) }{ (-1)^i \+ x } \biggr] \frac{ T c^{ 2 \delta - 1 } 2^{ p ( 2 \delta - 1 ) } N^{ p ( 2 \delta - 1 ) + 1 } }{ [ p ( 1 - 2 \delta ) - 1 ]  \exp \bigl( \tfrac{ 6 \+ p ( 2 \delta - 1 ) T c^{ 2 \delta - 1 } }{ p ( 2 \delta - 1 ) + 1 } \bigr) } > 0.
\end{equation}
\end{cor}
\begin{proof}[Proof of Corollary~\ref{cor:lower_bound_exp}]
Lemma~\ref{lem:L2norm}, Lemma~\ref{lem:sum}, and the fact that $ \forall \tp x \in (0,\infty) \colon \abs[\big]{ \frac{ \sin( x ) }{ x } } < 1 $ prove for all $ i \in \{ 1, 2 \} $, $ N \in \N $ that
\begin{equation} \label{eq:lower_bound_exp,2}
\begin{split}
& \E \Bigl[ \norm{ X^{ \bbH, i } }_{ H_{ \nicefrac{1}{2} - \nicefrac{i}{2} } }^2 \Bigr] - \E \Bigl[ \norm{ X^{ I, i } }_{ H_{ \nicefrac{1}{2} - \nicefrac{i}{2} } }^2 \Bigr] = \frac{1}{2} \sum_{ h \in \bbH \setminus I_N } \frac{ \abs{ \mu_h }^2 }{ \abs{ \lambda_h } } \biggl( T + \frac{ \sin \bigl( 2 \abs{ \lambda_h }^{ \nicefrac{1}{2} } T \bigr) }{ (-1)^i \+ 2 \abs{ \lambda_h }^{ \nicefrac{1}{2} } } \biggr) \\
& \geq \biggl( 1 + \inf_{ h \in \bbH } \frac{ \sin \bigl( 2 \abs{ \lambda_h }^{ \nicefrac{1}{2} } T \bigr) }{ (-1)^i \+ 2 \abs{ \lambda_h }^{ \nicefrac{1}{2} } T } \biggr) \frac{T}{2} \sum_{ h \in \bbH \setminus I_N } \abs{ \lambda_h }^{ 2 \delta - 1 } \\
& \geq \biggl( 1 + \inf_{ x \in [ 2 c^{\nicefrac{1}{2} } T, \infty ) } \frac{ \sin( x ) }{ (-1)^i \+ x } \biggr) \frac{T c^{ 2 \delta - 1 } }{2} \sum_{ n = N+1 }^\infty n^{ p ( 2 \delta - 1 ) } \\
& \geq \biggl( 1 + \inf_{ x \in [ 2 c^{\nicefrac{1}{2} } T, \infty ) } \frac{ \sin( x ) }{ (-1)^i \+ x } \biggr) \frac{ T c^{ 2 \delta - 1 } 2^{ p ( 2 \delta - 1 ) } N^{ p ( 2 \delta - 1 ) + 1 } }{ [ p ( 1 - 2 \delta ) - 1 ] } > 0.
\end{split}
\end{equation}
Furthermore, the assumption that $ \delta \in ( - \infty, \nicefrac{1}{2} - \nicefrac{1}{ (2 p ) } ) $ ensures that $ p ( 2 \delta - 1 ) \in ( - \infty, - 1 ) $. Hence, we obtain that
\begin{equation} \label{eq:lower_bound_exp,1}
\begin{split}
\sum_{ n = 1}^\infty n^{ p ( 2 \delta - 1 ) } & \leq 1 + \sum_{ n=1 }^{ \infty } \int_{ n }^{ n + 1 } x^{ p ( 2 \delta - 1 ) } \dx = 1 + \int_{1}^{ \infty } x^{ p ( 2 \delta - 1 ) } \dx \\
& = 1 - \frac{ 1 }{ p ( 2 \delta - 1 ) + 1 } = \frac{ p ( 2 \delta - 1 ) }{ p ( 2 \delta - 1 ) + 1 }.
\end{split}
\end{equation}
Lemma~\ref{lem:L2norm} and \eqref{eq:lower_bound_exp,1} imply for all $ i \in \{ 1, 2 \} $ that
\begin{equation}
\begin{split}
\exp \Bigl( - 6 \E \Bigl[ \norm{ X^{ \bbH, i } }_{ H_{ \nicefrac{1}{2} - \nicefrac{i}{2} } }^2 \Bigr] \Bigr) & \geq \exp \bigl( - 6 \E \bigl[ \norm{ \bfX^\bbH }_{ \bfH_0 }^2 \bigr] \bigr) = \exp \biggl( - 6 \+ T \+ c^{ 2 \delta - 1 } \sum_{ n = 1}^\infty n^{ p ( 2 \delta - 1 ) } \biggr) \\
& \geq \exp \biggl( - \frac{ 6 \+ p ( 2 \delta - 1 ) T c^{ 2 \delta - 1 } }{ p ( 2 \delta - 1 ) + 1 } \biggr) > 0.
\end{split}
\end{equation}
Combining this and \eqref{eq:lower_bound_exp,2} with Proposition~\ref{prop:lower_bound_exp} concludes the proof of Corollary~\ref{cor:lower_bound_exp}.
\end{proof}
Corollary~\ref{cor:lower_bound_laplacian} below specifies Corollary~\ref{cor:lower_bound_exp} to the case where the linear operator $ A \colon D(A) \subseteq H \to H $ in the setting in Section~\ref{subsec:setting_lower} is the Laplacian with Dirichlet boundary conditions on $ H $. It is an immediate consequence of Corollary~\ref{cor:lower_bound_exp}.
\begin{cor} \label{cor:lower_bound_laplacian}
Assume the setting in Section~\ref{subsec:setting_lower}, let $ \delta \in ( - \infty, \nicefrac{1}{4} ) $, let $ e \colon \N \to \bbH $ be a bijection which satisfies for all $ n \in \N $ that $ \lambda_{ e_n } = - \pi^2 n^2 $, let $ I_N \in \cP( \bbH ) $, $ N \in \N $, be the sets which satisfy for all $ N \in \N $ that $ I_N = \{ e_1, e_2, \ldots, e_N \} \subseteq \bbH $, assume for all $ h \in \bbH $ that $  \abs{ \mu_h } = \abs{ \lambda_{ h } }^\delta $, and let $ \varphi_i \colon \bfH_0 \to \R $, $ i \in \{ 1, 2 \} $, be the functions which satisfy for all $ i \in \{ 1, 2 \} $, $ ( v_1, v_2 ) \in \bfH_0 $ that \smash{$ \varphi_i( v_1 , v_2 ) = \exp \bigl( - \norm{ v_i }^2_{ H_{ \nicefrac{1}{2} - \nicefrac{i}{2} } } \bigr) $}. Then it holds for all $ i \in \{ 1, 2 \} $, $ N \in \N $ that $ \varphi_i \in C_{ \mathrm{b} }^2( \bfH_0, \R ) $ and
\begin{equation}
\E [ \varphi_i( \bfX^{ I_N } ) ] - \E[ \varphi_i( \bfX^\bbH ) ] \geq \biggl[ 1 + \inf_{ x \in [ 2 \pi T, \infty ) } \frac{ \sin( x ) }{ (-1)^i \+ x } \biggr] \frac{ T ( 4 \pi^2 )^{ 2 \delta - 1 } N^{ 4 \delta - 1 } }{ [ 1 - 4 \delta ]  \exp \bigl( \tfrac{ 12 ( 2 \delta - 1 ) T \pi^{ 4 \delta - 2 } }{ 4 \delta - 1 } \bigr) } > 0.
\end{equation}
\end{cor}

\section*{Acknowledgements}
This project has been partially supported through the ETH Research Grant \mbox{ETH-47 15-2}
``Mild stochastic calculus and numerical approximations for nonlinear stochastic evolution equations with L\'evy noise''.

\printbibliography

\end{document}